\documentclass{article}

\usepackage{amsmath,amssymb,amsfonts,amsthm}
\usepackage{booktabs}
\usepackage{graphicx}
\usepackage{url}
\usepackage{xcolor}
\allowdisplaybreaks

\renewcommand{\emptyset}{\varnothing}
\renewcommand{\le}{\leqslant}
\renewcommand{\ge}{\geqslant}
\renewcommand{\leq}{\leqslant}
\renewcommand{\geq}{\geqslant}

\newcommand{\real}{\mathbb{R}}
\newcommand{\ints}{\mathbb{Z}}
\newcommand{\natu}{\mathbb{N}}

\newcommand{\ind}{\mathbf{1}}

\newcommand{\rd}{\,\mathrm{d}}

\newcommand{\ip}{i'}

\newcommand{\aij}{a_{ij}}
\newcommand{\aipj}{a_{i'j}}

\newcommand{\bj}{b_j}

\newcommand{\bsa}{\boldsymbol{a}}
\newcommand{\bsb}{\boldsymbol{b}}
\newcommand{\bsk}{\boldsymbol{k}}
\newcommand{\bsr}{\boldsymbol{r}}
\newcommand{\bsx}{\boldsymbol{x}}

\newcommand{\e}{\mathbb{E}}

\newcommand{\var}{\mathrm{var}}

\newcommand{\bszero}{\boldsymbol{0}}

\newcommand{\modop}[2]{( #1\bmod #2)}

\newcommand{\err}{\varepsilon}

\newcommand{\dunif}{\mathbb{U}}
\newcommand{\simiid}{\stackrel{\mathrm{iid}}\sim}

\newcommand{\bvhk}{\mathrm{BVHK}}

\newtheorem{theorem}{Theorem}
\newtheorem{corollary}{Corollary}
\newtheorem{proposition}{Proposition}

\theoremstyle{definition}

\title{Gain coefficients for scrambled Halton points}
\author{Art B. Owen\\Stanford University
\and Zexin Pan\\Stanford University}
\date{August 2023}

\begin{document}
\maketitle

\begin{abstract}
  Randomized quasi-Monte Carlo, via certain 
  scramblings of digital nets, produces unbiased
  estimates  of $\int_{[0,1]^d}f(\bsx)\rd\bsx$
  with a variance that is $o(1/n)$ for any $f\in L^2[0,1]^d$.  
It also satisfies some non-asymptotic
  bounds where the variance is no larger than 
  some $\Gamma<\infty$ times the ordinary Monte Carlo
  variance.  For scrambled Sobol' points, this quantity
  $\Gamma$ grows exponentially in $d$.  For scrambled
  Faure points, $\Gamma \le \exp(1)\doteq 2.718$
  in any dimension, but those points are awkward to use
  for large $d$.  This paper shows that certain scramblings
  of Halton sequences have gains below an explicit bound that is $O(\log d)$ but not $O( (\log d)^{1-\epsilon})$ for any $\epsilon>0$ as $d\to\infty$. For $6\le d\le 10^6$, the upper bound on the gain coefficient is never larger than $3/2+\log(d/2)$.
\end{abstract}

\section{Introduction}

High dimensional integrals are often computed by plain
Monte Carlo (MC) sampling.  
In its basic form, we sample random vectors IID from their distribution,
evaluate some quantity of interest on the sampled
vectors and average the resulting values.
It is often possible to use a rich set of transformations
from $\dunif[0,1]^d$ (see \cite{devr:1986})
to generate the needed random vectors.  We can then write
the integral of interest as $\mu=\int_{[0,1]^d}f(\bsx)\rd\bsx$
and approximate it via $\hat\mu = (1/n)\sum_{i=0}^{n-1}f(\bsx_i)$
for $\bsx_i\simiid\dunif[0,1]^d$.

In quasi-Monte Carlo (QMC) 
sampling \cite{dick:pill:2010,nied92}, deterministic 
points $\bsx_i\in[0,1]^d$ are chosen
strategically to nearly minimize a measure of distance between the
discrete uniform distribution on $\{\bsx_0,\bsx_1,\dots,\bsx_{n-1}\}$
and the continuous uniform distribution on $[0,1]^d$.
Such distances are known as discrepancies \cite{chen:sriv:trav:2014}.  
The most widely studied
one is the star discrepancy $D_n^*(\bsx_0,\dots,\bsx_{n-1})$
which is a multivariate generalization of the Kolmogorov-Smirnov
distance between discrete and continuous uniform distributions.
It is possible to attain $D_n^* = O(\log(n)^{d-1}/n)$.
Then the Koksma-Hlawka inequality \cite{hick:2014}
yields $|\hat\mu-\mu| = O(n^{-1+\epsilon})$,
for any $\epsilon>0$,  when $f$
has bounded variation in the sense of Hardy and Krause,
which we write as $f\in\bvhk$.

While $\log(n)^{d-1}=O(n^\epsilon)$ for any $\epsilon>0$
it is natural to question whether that is
a good description for large $d$ and modest $n$.
Surprisingly, this expression seems reasonable for
applied work.
Those logarithmic powers apply for adversarially
chosen integrands $f$ that never seem to arise in
practice \cite{schl:2002} and it is challenging to construct
even one such integrand requiring a power of $\log(n)$ above $1$ 
\cite{wherearethelogs}, even  when exploiting
known weaknesses of some QMC constructions.

Some (but not all) randomized QMC (RQMC) methods provide stronger
assurances that high powers of $\log(n)$ do not correspond
to very bad accuracy.
In RQMC, one takes QMC points $\bsa_0,\dots,\bsa_{n-1}$ and
a random transformation $\tau$ such that $\bsx_i=\tau(\bsa_i)\sim\dunif[0,1]^d$
individually, while $\bsx_0,\dots,\bsx_{n-1}$
collectively have low discrepancy.
See \cite{lecu:lemi:2002} and \cite[Chapter 17]{practicalqmc}.
This allows us to get IID replicates $\hat\mu_r$
for $r=1,\dots,R$ that are unbiased for $\mu$
and we can use them to estimate the RQMC sampling variance.

Some RQMC methods give unbiased estimates of $\mu$ with variance
no larger than $\Gamma\sigma^2/n$ for some $\Gamma<\infty$
where $\sigma^2/n$ is the variance of $\hat\mu$ under IID
sampling. This bounds how much the powers of $\log(n)$ 
can make RQMC worse than plain MC which is the natural
default comparison for RQMC.  Also, if $f\in\bvhk$ then
$\var(\hat\mu) =O(n^{-2+\epsilon})$ for any $\epsilon>0$.

We take as our starting point, the nested uniform scrambling
of digital nets from \cite{rtms}.
That method provides an estimate $\hat\mu$
with many desirable properties noted in~\cite{sllnrqmc}.
It is unbiased: if $f\in L^1[0,1]^d$ then $\e(\hat\mu)=\mu$.
There is a strong law of large numbers:
if $f\in L^{1+\epsilon}[0,1]^d$ for some $\epsilon>0$ then
$\Pr( \lim_{n\to\infty} \hat\mu=\mu)=1$.
If $f\in L^2[0,1]^d$ then $\var(\hat\mu)=o(1/n)$.
If $f$ is sufficiently smooth, so that it has mixed partial derivatives
with respect to each input at most once that are in $L^2[0,1]^d$,
then $\var(\hat\mu) = O(n^{-3}(\log n)^{d-1})$.  The property of
most interest here is that if $f\in L^2[0,1]^d$, then there exists
$\Gamma<\infty$ such that $\var(\hat\mu)\le\Gamma\sigma^2/n$.
This quantity $\Gamma$ is called a `gain coefficient'.

The most popular QMC points are the digital nets
and sequences of Sobol' \cite{sobol67}.
They are constructed using dyadic (base $2$) representations
and are designed for sample sizes $n=2^m$.
The properties described above for RQMC can be attained
using either the nested uniform scrambling of \cite{rtms}
or the faster linear scrambling plus digital shift of \cite{mato:1998:2}.
Writing the original Sobol' points $\bsa_i=(a_{i1},\dots,a_{id})\in[0,1]^d$,
and then writing each $a_{ij}$ in terms of bits, the RQMC points
$\bsx_i$ are obtained by taking their bits to be certain randomizations
of the bits of $a_{ij}$.

For the purposes of this paper, the scrambled Sobol' points have
a disadvantage in that the value of $\Gamma$ for them grows
exponentially with dimension $d$.  In high dimensional settings,
an adversary that knew we
were about to use $n=2^m$ scrambled Sobol' points could choose an 
integrand $f\in L^2[0,1]^d$ 
where $\hat\mu$ would have much higher variance than under plain
Monte Carlo.  The worst case integrands are not smooth.
They are piecewise constant functions
over dyadic hyperrectangular subregions
of $[0,1]^d$ and they have rapidly alternating signs.  In many settings we can
be confident that these worst case integrands are extremely
unrealistic. Yet we might want a smaller value of $\Gamma$.

A smaller value of $\Gamma$ can be found by scrambling the
digital nets of Faure \cite{faures}.
While Sobol's points are constructed in base $2$, 
Faure's points are constructed in a more general
integer base $b\ge2$.
Scrambling the points of Faure, 
provides a bound of $\Gamma \le [b/(b-1)]^{d-1}$ in dimension $d$
\cite{owensinum}.
Because his construction requires $b\ge d$ it follows that the
maximal gain cannot exceed $\exp(1) \doteq 2.718$ in any dimension.
Faure's construction requires $b$ to be a prime number,
however it generalizes to the case where $b$ is a power of a prime \cite{nied:1987}.

Unfortunately, the point sets of Faure do not seem to do as well in
practice as those of Sobol'.  This can be explained by the fact that to
get nontrivial equidistribution in $s$-dimensional marginal projections
of $[0,1]^d$ they require at least $b^s$ points to be used.
Because $b\ge d$,  we then need to use $n\ge d^s$ points to
gain an appreciable advantage over plain MC in averaging
the $s$-dimensional interactions in an ANOVA decomposition of $f$.
QMC and RQMC points typically have very uniform $1$ dimensional
marginal projections $\{x_{0j},\dots,x_{n-1,j}\}$ and so the difficulties
with Faure points arise when $d^2$ or $d^3$ would be an uncomfortably
large value for $n$.

There is thus a gap.  How can we get RQMC constructions that converge
faster than those of Faure while having better upper bounds 
on $\Gamma$ than those of Sobol'?  
This article proposes scrambling of Halton points \cite{halt:1960} 
as a solution.  Halton points are less
commonly used than Sobol' points now, 
probably due to experience or beliefs
that Sobol' points provide greater accuracy.
Here, we show that Halton points have gain parameters that grow
at most slowly with dimension. Letting $\Gamma_d$ be the largest gain coefficient in $d$ dimensions, our main theoretical results are upper and lower bounds for $\Gamma_d$.
We easily find that $\Gamma_1=1$ and our bounds imply that 
\begin{align}\label{eq:mainbounds}
\frac34\prod_{j=1}^d\frac{b_j+1}{b_j}\le\Gamma_d \le \frac12\prod_{j=1}^d\frac{b_j}{b_j-1}
\end{align}
both hold for all $d\ge2$.
Using~\eqref{eq:mainbounds} we show that $\Gamma_d=O( \log d)$.
We also show that $\Gamma_d$ cannot be $O((\log d)^{1-\epsilon})$ for any $\epsilon>0$.
The bounds in~\eqref{eq:mainbounds} are shown in Figure~\ref{fig:bounds}. For $6\le d\le10^6$, the upper bound on $\Gamma_d$ never exceeds $3/2+\log(d/2)$, though that may fail to hold for some $d>10^6$.

\begin{figure}
\centering
\includegraphics[width=.9\hsize]{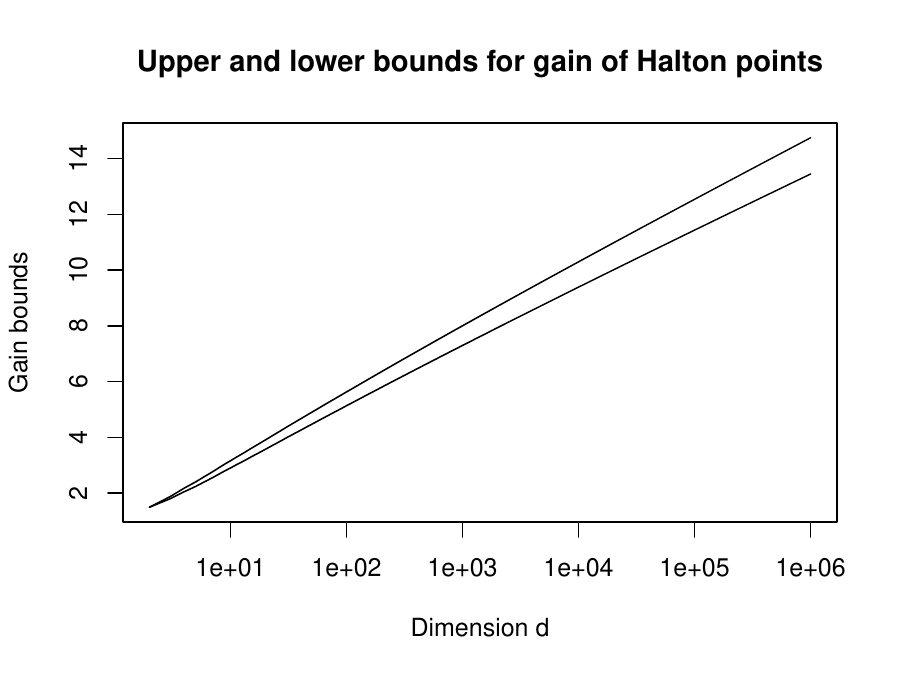}
\caption{\label{fig:bounds}
This figure shows the upper and lower bounds for $\Gamma_d$ from equation~\eqref{eq:mainbounds}. The horizontal axis is the dimension $d$ for $2\le d\le10^6$.
}
\end{figure}

This logarithmic rate for $\Gamma_d$ is much slower than the exponential
rate that Sobol' points have.  We might then prefer to use scrambled
Halton points in settings where we very much want to avoid
the worst outcomes even if it means less accuracy on benign cases.
Halton points are also easier to use than Faure points when $d$ is large.
If we rank the RQMC methods by worst case variances we
prefer Faure to Halton to Sobol'.  In high dimensional settings with
non-pathological integrands we might reasonably prefer the reverse
order.  Then Halton, coming second both times, may be a good
compromise choice.

The rest of this paper is organized as follows.
Section~\ref{sec:background} introduces some notation,
defines the Halton points and introduces gain coefficients
for all non-empty subsets of $s\le d$ variables and all vectors of 
$s$ nonnegative integers.
Section~\ref{sec:nonasymptotic} gives some expressions for gain
coefficients at special sample sizes $n$.
It also shows that the gain coefficients are $O(1/n)$ from which the scrambled Halton variance is $o(1/n)$ for any integrand in $L^2[0,1]^d$.
Section~\ref{sec:examples} has numerical examples to illustrate
how gain coefficients vary with $n$.
Section~\ref{sec:gaininequalities} has two theorems that identify precisely where the worst gain coefficients must lie and then establishes the upper bound in~\eqref{eq:mainbounds}.
Section~\ref{sec:lower} establishes the lower bound in~\eqref{eq:mainbounds}.
Section~\ref{sec:conclusions} has brief conclusions.

\section{Background material}\label{sec:background}

\subsection{Basic notation}
We use $\real$ for the real numbers,
$\ints$ for the integers, $\natu$ for the positive
integers, $\natu_0=\natu\cup\{0\}$
and $\ints_m=\{0,1,\dots,m-1\}$
for $m\in\natu$. 
We use $1{:}d$ to denote $\{1,2,\dots,d\}$.

For $u\subseteq1{:}d$, we use $|u|$
for the  cardinality of $u$ and $-u$
for the complementary set $1{:}d\setminus u$. 
A vector of zeros is denoted by $\bszero$.
If $u=\{j_1,j_2,\dots,j_{|u|}\}$ then
we use $\natu_0^u$
to denote a copy of $\natu_0^{|u|}$
that can be indexed by the elements of $u$.
For example, from any $\bsk\in\natu_0^{\{1,2,4\}}$
we can obtain components $k_1$, $k_2$ and $k_4$.

For $z\in\real$, we let $\lfloor z\rfloor=\max\{y\in\ints\mid y\le z\}$.
For $a\in\natu_0$ and $b\in \natu$ the residue of $a$ modulo $b$
is $a-\lfloor a/b\rfloor b$ which we denote by $\modop{a}{b}$.

The expressions $\ind_A$ and $\ind\{A\}$ are both
indicators, taking the value $1$ when $A$ holds and
$0$ when $A$ does not hold.  The choice of which
to use is made based on readability. 
%which can be affected by the presence
%of subscripts and/or superscripts in the expression for $A$.

\subsection{Halton points}
Let $b_j$ be the $j$'th largest prime number for $j\in\natu$.
The base $b_j$ digits of $i\in\natu_0$ are denoted $a_{ij\ell}$.
That is, for $i\in\natu_0$, and $j\in\natu$, we can write
$$
i = \sum_{\ell=1}^\infty a_{ij\ell}b_j^{\ell-1}
$$
for $a_{ij\ell}\in \ints_{b_j}$.  This sum has only
finitely many nonzero terms for any $i\in\natu_0$.
The unscrambled Halton points are $\bsa_i\in[0,1)^d$ for $i\in\natu$ with
\begin{align}\label{eq:defaij}
a_{ij} = \sum_{\ell=1}^\infty a_{ij\ell}b_j^{-\ell}
\end{align}
for $j\in1{:}d$.  Halton points can be defined by taking $b_j$ to
be any $d$ relatively prime natural numbers. In practice, the
first $d$ primes are almost always used and we will work with
that assumption.

Here is a brief intuitive description of why Halton points
fill the unit cube nearly uniformly. For more details see \cite{halt:1960}.
For $j=1$, as integers $i$ alternate between even and odd,
the first digit $a_{i11}$ alternates between $0$ and $1$
and then the point $a_{i1}$ alternates between being
in $[0,1/2)$ and $[1/2,1)$ so we always have nearly
half of the points in $[0,1/2)$ and half in $[1/2,1)$.
More generally, any consecutive $2^k$ integers $i$
contain all values of in $\ints_{2^k}$ and then
the corresponding $a_{i1}$ will be balanced
over $[r/2^k,(r+1)/2^k)$ for $r\in\ints_{2^r}$.
Still more generally, for $j\ge1$ and
any $b_j^{k_j}$ consecutive indices $i\in\natu_0$,
the values $a_{ij}$  stratify over 
$[r/b_j^{k_j},(r+1)/b_j^{k_j})$ for $r\in\ints_{b_j^{k_j}}$.
For $\bsk\in\natu_0^d$ we can
consider the Halton strata
\begin{align}\label{eq:haltonstrata}
S_{\bsr}(\bsk)=\prod_{j=1}^d\Bigl[ \frac{r_j}{b_j^{k_j}}, \frac{r_j+1}{b_j^{k_j}}\Bigr)
\end{align}
with $r_j\in\ints_{b_j^{k_j}}$.
By the Chinese remainder theorem,
every consecutive batch of $\prod_{j=1}^db_j^{k_j}$ points
has exactly one member in each of the strata above.  
Any subsequent batch of fewer than $\prod_{j=1}^db_j^{k_j}$
points
is spread through those strata, with at most one of them 
in each stratum.
Smaller bases $b_j$ tend to provide better
equidistribution properties than larger bases do.
As a result, when using Halton points, it can be very
valuable to arrange for the most important input
variables to have the lowest indices.  A perfect
definition of variable importance would be tautological
and not very helpful.
In practice, one can use scientific understanding/intuition
or proxy measures such as Sobol' indices
\cite{dave:gamb:ioos:prie:2021} to order the inputs.

% Equidistribution tends to be better for smaller $b_j$
% so .... put the most important variables first
% The exact definition of important would be tautological
% A practical approach is 

While Halton points are asymptotically equidistributed,
it is well known that for small $n$ and large $d$,
the points tend to show unwanted structure.
For $i<100$, $a_{i,26}=(i\bmod 101)/101$ and 
$a_{i,27}=(i\bmod 103)/103$ are collinear.
There have been many proposals to break up this
unwanted structure by, for example, replacing
$a_{ij\ell}$ in~\eqref{eq:defaij} by some permuted values
$\pi(a_{ij\ell})$ where $\pi(\cdot)$ can depend on 
$j$ and $\ell$.
There are deterministic proposals
in \cite{braa:well:1979}, \cite{faur:1992} and \cite{vand:cool:2006}
and others described in \cite{faur:lemi:2009}.
There is a random permutation proposal in~\cite{okte:shah:gonc:2012}
with a study and implementation in \cite{rhalton} and
another kind of randomization in \cite{wang:hick:2000}.

Here we consider two randomizations.  One is the nested uniform
scramble \cite{rtms} in base $b_j$ applied to the $j$'th component
of $\bsa_i$ with all $d$ randomizations statistically independent of each other. The other is the random linear scramble, with digital shift,
from \cite{mato:1998:2}.
Faure and Lemieux \cite{faur:lemi:2009} have considered the
linear scramble, without a digital shift, for Halton points.
They did not use random scrambles but instead did a computer
search to find a scramble to recommend for general use.

\subsection{Gain coefficients}
Digital nets are similar to Halton points, except that they use
the same base $b$ for every component of the $n$ points.
Gain coefficients for scrambled digital nets
were presented in \cite{owensinum}.
They arise from a $d$-fold tensor product 
of a base $b$ Haar wavelet basis for $L^2[0,1]$.
For Halton points, we use instead a tensor product
of Haar wavelet basis functions with the $j$'th
one defined in terms of base $b_j$.
For non-empty $u\subseteq1{:}d$, $\bsk\in\natu_0^u$ and
integer $n\ge1$, define the gain coefficient
\begin{align}\label{eq:defgain}
\begin{split}
G_{u,\bsk}(n)&=
\frac1n\prod_{j\in u}(\bj-1)^{-1}\widetilde G_{u,\bsk}(n),
\quad\text{where}\\
\widetilde G_{u,\bsk}(n) &= \sum_{i=0}^{n-1}\sum_{\ip=0}^{n-1}
\prod_{j\in u}b_j
\ind_{\lfloor b_j^{k_j+1} \aij\rfloor=\lfloor b_j^{k_j+1}\aipj\rfloor}
-\ind_{\lfloor b_j^{k_j}\aij\rfloor=\lfloor b_j^{k_j}\aipj\rfloor}.
\end{split}
\end{align}
This formula is a generalization of the one in \cite[Theorem 2]{owensinum}
that uses the same base $b$ in every dimension. These gain coefficients apply to scrambling of arbitrary point sets, though they have useful simplifications for some quasi-Monte Carlo points.

Each $f\in L^2[0,1]^d$ has variance components
$\sigma^2_{u,\bsk}$ defined through the wavelet basis.
The variance $\sigma^2$ of $f$ satisfies
$$
\sigma^2 = \sum_{u\subseteq 1:d}\sum_{\bsk\in\natu_0^u}\sigma^2_{u,\bsk}.
$$
We take $\sigma^2_{\emptyset,()}=0$ because it corresponds to a
constant term which does not contribute to the sampling variance.
If we use $n\ge1$ randomized Halton points then the estimate
$$
\hat\mu_n=\frac1n\sum_{i=0}^{n-1}f(\bsx_i)
$$
is an unbiased estimate of $\mu=\int_{[0,1]^d}f(\bsx)\rd\bsx$
with variance
$$
\frac1n\sum_{u\subseteq1{:}d}\sum_{\bsk\in\natu_0^u}G_{u,\bsk}(n)\sigma^2_{u,\bsk}
\le \frac{\Gamma_d(n)\sigma^2}{n}
$$
where
$$
\Gamma_d(n) = \max_{u\subseteq1:d}\sup_{\bsk\in\natu_0^u}G_{u,\bsk}(n).
$$
Estimation using scrambled Halton points cannot 
have more than $\Gamma_d(n)$
times the variance from using plain Monte Carlo points.
It is then interesting to bound $\Gamma_d(n)$.
We will also get a bound for
$$
\Gamma_d = \sup_{n\in\natu}\Gamma_d(n).
$$

% Additionally, for any $u$ and $\bsk$,
% $$
% \lim_{n\to\infty} G_{u,\bsk}(n)=0.
% $$
% This suggests that $\var(\hat\mu) = o(1/n)$ for
% any $f\in L^2[0,1]^d$, but we want to
% show that $\sup_{\bsk\in\natu_0^{u}}G_{u,\bsk}(n)<\infty$ 
% for sufficiently large $n$ to be sure.

\subsection{Preliminary results}

Here we present some elementary results 
to simplify some of the derivations 
for gain coefficients.
We begin by defining two quantities that frequently
arise in our expressions.
For non-empty $u\subseteq1{:}d$ and any $v\subseteq u$, let
\begin{align}\label{eq:defhu}
H_{u,v}= \prod_{j\in v}b_j\prod_{j\in u-v}(-1).
\end{align}
Then for $\bsk\in\natu_0^{u}$  define
\begin{align}\label{eq:defmuvk}
  m_{u,v,\bsk} = \prod_{j\in v}b_j^{k_j+1}\prod_{j\in u-v}b_j^{k_j}.
\end{align}

By inclusion-exclusion, we may write
\begin{align*}
&\phantom{=}\,\,\prod_{j\in u}
\Bigl(b_j
\ind_{\lfloor b_j^{k_j+1}\aij\rfloor=\lfloor b_j^{k_j+1}\aipj\rfloor}
-\ind_{\lfloor b_j^{k_j}\aij\rfloor=\lfloor b_j^{k_j}\aipj\rfloor}\Bigr)\\
&=
\sum_{v\subseteq u}
H_{u,v}
\prod_{j\in v}
\ind_{\lfloor b_j^{k_j+1}\aij\rfloor =\lfloor b_j^{k_j+1}\aipj\rfloor}
\prod_{j\in u-v}\ind_{\lfloor b_j^{k_j}\aij\rfloor =\lfloor b_j^{k_j}\aipj\rfloor}.
\end{align*}
For $\aij$ given by~\eqref{eq:defaij} and $r\ge0$,
$$\lfloor b_j^r\aij\rfloor 
= \bigg\lfloor\sum_{\ell=1}^\infty b_j^{r-\ell}a_{ij\ell}\bigg\rfloor
= \sum_{\ell=1}^r b_j^{r-\ell}a_{ij\ell}.$$
Therefore
$\lfloor b_j^r\aij\rfloor=\lfloor b_j^r\aipj\rfloor$
if and only if
$$\sum_{\ell=1}^r b_j^{r-\ell}a_{ij\ell}=\sum_{\ell=1}^r b_j^{r-\ell}a_{\ip j\ell}$$
which holds if and only
$i = \ip \bmod b_j^r$.
Then using the Chinese remainder theorem
\begin{align}\label{eq:crt}
  &\phantom{=}\,\,
\prod_{j\in v}\ind_{\lfloor b_j^{k_j+1}\aij\rfloor =\lfloor b_j^{k_j+1}\aipj\rfloor}
\prod_{j\in u-v}\ind_{\lfloor b_j^{k_j}\aij\rfloor =\lfloor b_j^{k_j}\aipj\rfloor}\notag\\
  &=
\prod_{j\in v}\ind\{i=i'\bmod b_j^{k_j+1}\}
\prod_{j\in u-v}\ind\{i=i'\bmod b_j^{k_j}\}\notag\\
  &=
\ind\{i=i'\bmod m_{u,v,\bsk}\}.
\end{align}

For $m,n\in\natu$ let
\begin{align}\label{eq:defcmn}
C_{m,n}=\sum_{i=0}^{n-1}\sum_{\ip=0}^{n-1}\ind\{i=\ip\bmod m\}.
\end{align}
Then the unnormalized gain coefficients from~\eqref{eq:defgain} satisfy
\begin{align}\label{eq:GinHC}
    \widetilde G_{u,\bsk}(n) = \sum_{v\subseteq u}H_{u,v}C_{m_{u,v,\bsk},n}
= \sum_{v\subseteq u}H_{u,v}C_{u,v,\bsk}(n)
\end{align}
where $C_{u,v,\bsk}(n)$ is a more readable replacement
for $C_{m_{u,v,\bsk},n}$.

\begin{proposition}\label{prop:cmn}
  For $m,n\in\natu$,
\begin{align}\label{eq:cmn}
C_{m,n} = n +(2n-m)\lfloor n/m\rfloor -m\lfloor n/m\rfloor^2.
\end{align}
\end{proposition}
\begin{proof}
Write $n = mq+r$ for quotient $q=\lfloor n/m\rfloor\in\natu_0$ and remainder $r\in\ints_m$.
Then as explained below,
\begin{align*}
C_{m,n} 
&= mq^2+(2q+1)r\\
&= mq^2+(2q+1)(n-mq)\\
&= m\lfloor n/m\rfloor^2+(2\lfloor n/m\rfloor+1)(n-m\lfloor n/m\rfloor)\\
&= n +(2n-m)\lfloor n/m\rfloor -m\lfloor n/m\rfloor^2.
\end{align*}
The $mq^2$ term comes from $\sum_{i=0}^{mq-1}\sum_{i'=0}^{mq-1}\ind\{i=i'\bmod m\}$.
We get $qr$ from
$\sum_{i=0}^{mq-1}\sum_{i'=mq}^{mq+a-1}\ind\{i=i'\bmod m\}$
and another $qr$ with the indices $i$ and $i'$ reversed.
Finally, $\sum_{i=mq}^{mq+r-1}\sum_{i'=mq}^{mq+r-1}\ind\{i=i'\bmod m\}=r$.
\end{proof}

% old section remains in file oldarticle.tex
We may write the fractional part
  $\lfloor n/m\rfloor$ arising in $C_{m,n}$
  by $n/m-\err$ for some $0\le \err\le 1$,
  for each $m=m_{u,v,\bsk}$.
Doing this we get
\begin{align}\label{eq:thatotherapproach}
C_{u,v,\bsk}(n) = \frac{n^2}{m_{u,v,\bsk}}+m_{u,v,\bsk}\err_v(1-\err_v)
\end{align}
where $0\le\err_v\le1$,
which we will use later.
\section{Non-asymptotic results}\label{sec:nonasymptotic}

Here we show some non-asymptotic properties of the gain coefficients.  We also show that for scrambled Halton points $\var(\hat\mu)=o(1/n)$ when $f\in L^2[0,1]^d$.

Let $\underbar{m}_{u,\bsk}=m_{u,\emptyset,\bsk}$
and $\overline{m}_{u,\bsk}=m_{u,u,\bsk}$.
These are the minimal and maximal values of $m_{u,v,\bsk}$,
respectively.
We assume throughout that $u\ne\emptyset$.

\begin{proposition}\label{prop:gainone}
  If $1\le n< \underline{m}_{u,\bsk}$ then
  $$
G_{u,\bsk}(n) = 1.
  $$
  \end{proposition}
\begin{proof}
If $n<\underbar{m}_{u,\bsk}=m_{u,\emptyset,\bsk}$,
then $\lfloor n/m_{u,v,\bsk}\rfloor=0$ and
from~\eqref{eq:cmn}, $C_{u,v,\bsk}(n) = n.$
In this case
\begin{align*}
\widetilde G_{u,\bsk}=\sum_{v\subseteq u}H_{u,v}C_{u,v,\bsk}(n)
&=n\sum_{v\subseteq u}\prod_{j\in v}b_j\prod_{j\in u-v}(-1)
=n\prod_{j\in u}(b_j-1).
\end{align*}
Therefore $G_{u,\bsk}=1$, because the gain coefficients
in~\eqref{eq:defgain}
are defined with a normalizing factor of $\prod_{j\in u}(b_j-1)^{-1}/n$.
\end{proof}

\begin{proposition}\label{prop:gainzero}
  If $n=r\overline{m}_{u,\bsk}$ for $r\in\natu$, then
  $$
G_{u,\bsk}(n) = 0.
  $$
  \end{proposition}
\begin{proof}
If $n=r\overline{m}_{u,\bsk}$, for $r\in\natu_0$, then
for all $v\subseteq u$,
\begin{align*}
C_{u,v,\bsk}(n) &= n+(2n-m_{u,v,\bsk})(n/m_{u,v,\bsk})-m_{u,v,\bsk}(n/m_{u,v,\bsk})^2\\
&=n^2/m_{u,v,\bsk}.
\end{align*}
Now
\begin{align*}
\sum_{v\subseteq u}H_{u,v}\frac{n^2}{m_{u,v,\bsk}}
&=n^2\sum_{v\subseteq u}
\Biggl[\,\prod_{j\in v}b_j\prod_{j\in u-v}(-1)\Biggr]
\prod_{j\in u}b_j^{-k_j}\prod_{j\in v}b_j^{-1}\\
&=\frac{n^2}{m_{u,\emptyset,\bsk}}\sum_{v\subseteq u}
(-1)^{|u-v|}=0
\end{align*}
and so $G_{u,\bsk}(n)=0$ by equation~\eqref{eq:GinHC}.
\end{proof}

A gain of zero is the expected result.
For such $n$ we have attained zero discrepancy 
for all of the Halton strata congruent to $\prod_{j\in u}[0,1/b_j^{k_j+1})\prod_{j\in -u}[0,1)$.
There are $\overline{m}_{u,\bsk}$ such
strata defined by $u$ and $\bsk$, and
so $G_{u,\bsk}(n)$ cannot be zero for $n<\overline{m}_{u,\bsk}$.
Next we show that $G_{u,\bsk}(n)$ cannot re-attain its
maximal value for any $n>\overline{m}_{u,\bsk}$.

\begin{proposition}\label{prop:notagain}
Let $n = q\overline{m}_{u,\bsk}+r$ for 
$q\in\natu$ and
$r\in\ints_{\overline{m}_{u,\bsk}}\setminus\{0\}$.
Then
\begin{align}\label{eq:notagain}
G_{u,\bsk}(n) = \frac{r}{n}G_{u,\bsk}(r).
\end{align}
\end{proposition}
\begin{proof}
For any $i'\in\natu$ and any $r\in\ints_{\overline{m}_{u,\bsk}}$,
\begin{align*}
&\phantom{=}\,\,
\sum_{i=r}^{r+\overline{m}_{u,\bsk}-1}
\prod_{j\in u}b_j
\ind_{\lfloor b_j^{k_j+1}\aij\rfloor   =\lfloor b_j^{k_j+1}\aipj\rfloor}
-\ind_{\lfloor b_j^{k_j}\aij\rfloor   =\lfloor b_j^{k_j}\aipj\rfloor}\\
&=\sum_{i=r}^{r+\overline{m}_{u,\bsk}-1}
\sum_{v\subseteq u}H_{u,v}
\ind_{\lfloor b_j^{k_j+1}\aij\rfloor   =\lfloor b_j^{k_j+1}\aipj\rfloor}
\times\ind_{\lfloor b_j^{k_j}\aij\rfloor   =\lfloor b_j^{k_j}\aipj\rfloor}\\
&=\sum_{v\subseteq u}H_{u,v}
\sum_{i=r}^{r+\overline{m}_{u,\bsk}-1}\ind_{i=\ip\bmod m_{u,v,\bsk}}\\
&=\sum_{v\subseteq u}H_{u,v}\prod_{j\in u-v}b_j\\
&=0.
  \end{align*}
  The last step follows by the 
 argument used in the proof of Proposition~\ref{prop:gainzero}.
If $r>0$, then
\begin{align*}
\widetilde G_{u,\bsk}(n) &= \sum_{i=0}^{r-1}\sum_{i'=0}^{r-1}
\prod_{j\in u}b_j\ind_{\lfloor b_j^{k_j+1}\aij\rfloor   =\lfloor b_j^{k_j+1}\aipj\rfloor}
-\ind_{\lfloor b_j^{k_j}\aij\rfloor   =\lfloor b_j^{k_j}\aipj\rfloor}\\
&= \widetilde G_{u,\bsk}(r).
\end{align*}
Now~\eqref{eq:notagain} follows by the normalization
in~\eqref{eq:defgain}.
\end{proof}

We left the case $r=0$ out of Proposition~\ref{prop:notagain}.
We know that $G_{u,\bsk}(n)=0$ in
that case.  However we have not chosen
a convention for $G_{u,\bsk}(0)$.
We think that $G_{u,\bsk}(0)=1$ is reasonable
since $n=0$ for RQMC is the same as $n=0$
for MC, but we have not found another need for such
a convention.

\begin{corollary}\label{cor:nvartozero}
  If $f\in L^2[0,1]^d$ and $\bsx_0,\dots,\bsx_{n-1}$
  are points of a Halton sequence randomized
  with a nested uniform scramble, or a random linear
  scramble with digital shift, then
  $$
  \lim_{n\to\infty} n\cdot\var\biggl(
  \frac1n\sum_{i=0}^{n-1}f(\bsx_i)\biggr) =0.
  $$
  \end{corollary}
  \begin{proof}
    Let $f$ have variance components $\sigma^2_{u,\bsk}$.
    Then
    \begin{align*}
n\cdot\var(\hat\mu) &=\sum_{u\subseteq1{:}d}\sum_{\bsk\in\natu_0^u}G_{u,\bsk}(n)\sigma^2_{u,\bsk}\\
&=\sum_{u\subseteq1{:}d}\sum_{\bsk\in\natu_0^u}
\frac{ n\bmod \overline{m}_{u,\bsk}}nG_{u,\bsk}(n)\sigma^2_{u,\bsk}\\
&\le\Gamma_d\sum_{u\subseteq1{:}d}\sum_{\bsk\in\natu_0^u}
\frac{ n\bmod \overline{m}_{u,\bsk}}n\sigma^2_{u,\bsk}\\
&\to 0
    \end{align*}
as $n\to\infty$.
  \end{proof}

The next proposition shows that any values of
$G_{u,\bsk}(n)$ reappear as values of
$G_{u,\bsk'}(n')$ where $\bsk'$ is any vector
in $\natu_0^u$ no smaller than $\bsk$ componentwise
and $n'$ is some value $n'\ge n$.

\begin{proposition}\label{prop:badside}
For $j\in u\subseteq1{:}d$ and $\bsk\in\natu_0^u$
define $\bsk'$ by $k'_j=k_j+1$ and
$k'_\ell = k_\ell$ for $\ell \in u-\{j\}$.
Then 
\begin{align}\label{eq:badside}
G_{u,\bsk'}(nb_j)=G_{u,\bsk}(n)
\end{align}
for all $n\in\natu$.
\end{proposition}
\begin{proof}
First
\begin{align*}
\widetilde G_{u,\bsk'}(b_jn) = 
\sum_{v\subseteq u}H_{u,v}C_{u,v,\bsk'}(b_jn).
\end{align*}
Now
\begin{align*}
C_{u,v,\bsk'}(b_jn)
&= b_jn + (2b_jn-m_{u,v,\bsk'})
\lfloor nb_j/m_{u,v,\bsk'}\rfloor - m_{u,v,\bsk'}\lfloor
b_jn/m_{u,v,\bsk'}\rfloor^2\\
&= b_jn + (2b_jn-b_jm_{u,v,\bsk})
\lfloor n/m_{u,v,\bsk}\rfloor - b_jm_{u,v,\bsk}\lfloor
n/m_{u,v,\bsk}\rfloor^2\\
&= b_jC_{u,v,\bsk}(n).
\end{align*}
It follows that $\widetilde G_{u,\bsk'}(b_jn)=b_j\widetilde G_{u,\bsk}$.
Then~\eqref{eq:badside} holds after normalization.
\end{proof}

\begin{corollary}\label{cor:notenough}
  For nonempty $u\subseteq1{:}{d}$ and $\bsk\in\natu_0^u$,
  $$
G(u,\bsk)\biggl(n\prod_{j\in u}b_j^{k_j}\biggr)=G(u,\bszero)(n)
  $$
holds for all $n\ge1$.
\end{corollary}
\begin{proof}
We make  $\sum_{j\in u}k_j$ applications
of Proposition~\ref{prop:notagain}.
\end{proof}

\section{Example computations}\label{sec:examples}

It is straightforward to compute the gain coefficients
for scrambled Halton points in some settings of interest.
Figure~\ref{fig:gains23} shows the gain coefficients in the
smallest interesting case: $d=2$ and $\bsb = (2,3)$
for $1\le n\le 36$.  We see that all
$\bsk\in\{(0,0),(0,1),(1,0),(1,0)\}$ attain the same maximal
gain factor of $3/2$.
All of the curves start at gain equal to one for $n=1$.
This makes sense because $n=1$ scrambled Halton
point is mathematically equivalent to $n=1$ Monte Carlo
point.
The curves are initially one for all  $n \le \prod_{j\in u}b_j^{k_j}$
(see Proposition~\ref{prop:gainone})
and then with some oscillation,
they reach zero at $n = \prod_{j\in u}b_j^{k_j+1}$
(see Proposition~\ref{prop:gainzero}).
After reaching
zero they keep oscillating, but they will never again (for any larger $n$) re-attain their maximum 
(see Proposition~\ref{prop:notagain}).
The curve for $\bsk$ attains its peak at $n=2\prod_{j\in u}b_j^{k_j}$.
The factor $\prod_{j\in u}b_j^{k_j}$ 
is in line with Proposition~\ref{prop:badside}.

\begin{figure}[t]
\centering
\includegraphics[width=.9\hsize]{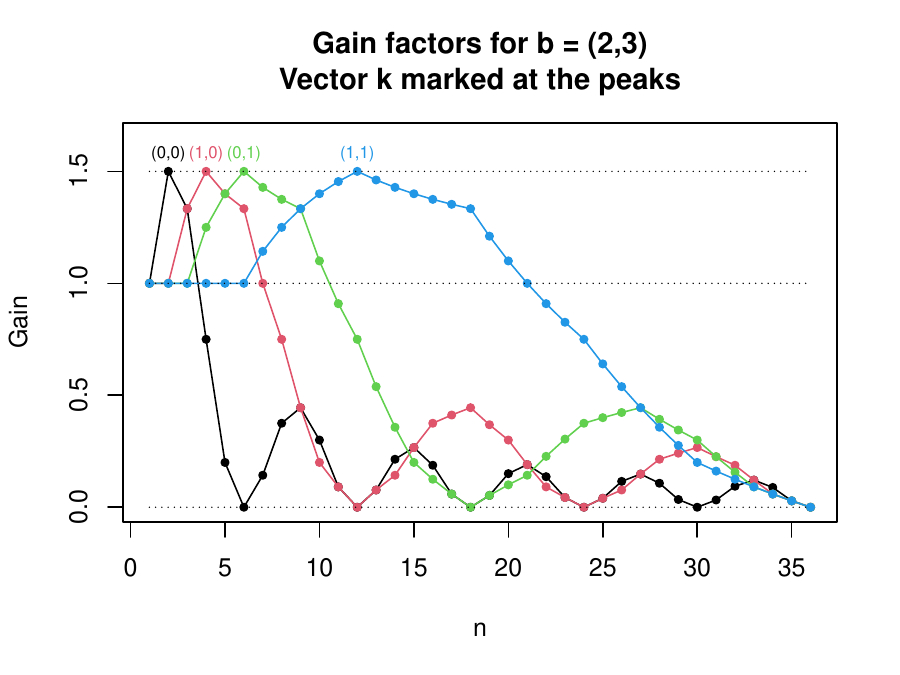}
\caption{\label{fig:gains23}
  For $d=2$ with $\bsb = (2,3)$ this figure shows the
  gains for $\bsk\in\{(0,0),(0,1),(1,0),(1,1)\}$ versus
  $1\le n\le 36$.  At $n=36$ all four of these gains are zero.
The same peak value $3/2$ attained by all curves.
In all cases, the maximum
is attained at $n=2\times b_1^{k_1}\times b_2^{k_2}$.
The horizontal reference lines are at gains $0$, $1$, and $3/2$.
  }
\end{figure}

Figure~\ref{fig:maxgains235} shows gain coefficients for
$d=3$ with $\bsb = (2,3,5)$.  The values of $n$ range
from $1$ to $1000$. Vectors $\bsk$ with 
$\prod_{j\in u}b_j^{k_j}>1000$ have gain $1$ for all $n$ in this range.
The plot shows gain curves for all other vectors $\bsk$.
It is clear that any value of $n$ has a maximal gain close
to the overall maximum (empirically $9/8$).  
In this worst case sense, the scrambled
Halton points do not have especially good values of $n$.
In another sense, described next, there do exist especially
good values of~$n$.

\begin{figure}[t]
\centering
\includegraphics[width=.9\hsize]{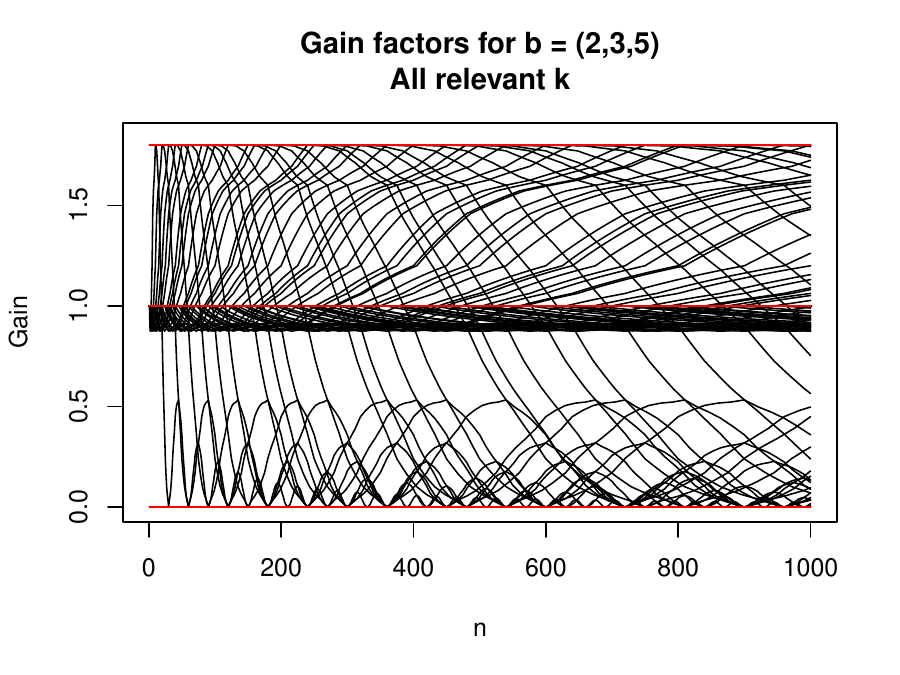}
\caption{\label{fig:maxgains235}
For $d=3$ with $\bsb = (2,3,5)$
this figure shows gain factors $G_{u,\bsk}(n)$ 
versus $n$ for all
  non-empty $u\subseteq\{1,2,3\}$ and all
  $\bsk$ with $\prod_{j\in u}b_j^{k_j}<n$.
  For any other $\bsk$ we know that $G_{u,\bsk}(n)=1$
  over this range for $n$.  There are horizontal reference lines
  at gains $0$, $1$, and $9/5$.
  }
\end{figure}

If we anticipate that smaller values of $|u|$ and of
$\prod_{j\in u}b_j^{k_j}$ correspond to more important features
of the function, then values of $n$ that are divisible by
products of small powers of the $b_j$ have an advantage.
We see in Figure~\ref{fig:gains23} that special values of $n$
give gain equal to zero for some of the effects with
small $\bsk$.
From Figure~\ref{fig:maxgains235} we can see that 
selecting such special value of $n$ will not give a meaningful
penalty with regard to worst case behavior.  This leaves us
more free to use convenient or highly
composite values of $n$.
Values of $n$ that are powers of $10$ are often popular with
users.  For the Halton sequence, such $n$ are very good for
the first and third input dimensions. 
A value like $n=1800=2^33^25^2$ can be expected to give
good results when the integrand depends strongly on
the first three components of $\bsx$ in a smooth way.
A user who wants $n$ to be a power of $10$ might
then use bases $2$ and $5$ for what they think are
most and second most important input variables, respectively.

A striking feature of Figure~\ref{fig:maxgains235} is a thick band between gains of 1 and 7/8. The latter value is $G_{1{:}3,\bszero}(2)$. The gains for every $\bsk$ decrease from $1$ to $7/8$ before rising to $9/5$.

In  Figures~\ref{fig:gains23} and~\ref{fig:maxgains235}
we never see any $G_{u,\bsk}(n)
>\max_{n\in\natu}G_{u,\bszero}(n)$.  
Theorem~\ref{thm:onlykzero} in Section~\ref{sec:gaininequalities} proves that this can never happen. Theorem~\ref{thm:supersetswin} in Section~\ref{sec:gaininequalities} shows that if $v\subsetneq u$ then  $\sup_{n\ge1}G_{v,\bszero}(n)\le \sup_{n\ge1}G_{u,\bszero}(n)$. Therefore the largest gains for $d$ variables arise in $G_{1:d,\bszero}(n)$ and we only need to consider $n$ from $1$ to $\prod_{j=1}^db_j$.

\section{Upper bounds for gain}\label{sec:gaininequalities}

It is of interest to know the largest possible values of gain coefficients.  Here, Theorem~\ref{thm:onlykzero} shows that we only need to consider $\bsk=\bszero$.  Then Theorem~\ref{thm:supersetswin} shows that we only need to consider $u=1{:}d$.  Applying Proposition~\ref{prop:notagain}, the largest possible gain for $d\ge1$ is one of $G_{1{:}d,\bszero}(n)$ for $1\le n\le\prod_{j=1}^db_j$.

\begin{theorem}\label{thm:onlykzero}
    For all  $1\le d<\infty$ and all nonempty $u\subseteq1{:}d$
and all $\bsk\in\natu_0^u$,
\begin{align}
\sup_{n\in\natu} G_{u,\bsk}(n) = \sup_{n\in\natu}G_{u,\bszero}(n).
\end{align}
\end{theorem}

\begin{proof}
Let $b^*=\prod_{j\in u}b_j^{k_j}$.
    Corollary~\ref{cor:notenough} shows that
      \begin{align}\label{eq:fromcornotenoug}
G_{u,\bsk}(nb^*)=G_{u,\bszero}(n).
\end{align}
 It suffices to show that for $n'$ such that $nb^*\leq n'\leq (n+1)b^*$, $G_{u,\bsk}(n')$ is maximized at the endpoints. That  is, we will show that
  $$\sup_{nb^*\leq n'\leq (n+1)b^*}G_{u,\bsk}(n')=\max\bigl(G_{u,\bsk}(nb^*),G_{u,\bsk}((n+1)b^*)\bigr)
$$  %\leq \sup_{n\in\natu}G_{u,\bszero}(n).$$
which is at most $\sup_{n\in\natu}G_{u,\bszero}(n)$
by~\eqref{eq:fromcornotenoug}.
By equations~\eqref{eq:GinHC} and \eqref{eq:thatotherapproach},
\begin{equation}\label{eq:Gexpr}
    \widetilde G_{u,\bsk}(n') 
= \sum_{v\subseteq u}H_{u,v}m_{u,v,\bsk}\err_v'(1-\err_v')
\end{equation}
where $\err_v'=n'/m_{u,v,\bsk}-\lfloor n'/m_{u,v,\bsk}\rfloor$. We write $n'=nb^*+r$ for $0\leq r\leq b^*$. Because $m_{u,v,\bsk}=b^*m_{u,v,\bszero}$,
\begin{align*}
   \err_v'&=\frac{n'}{m_{u,v,\bsk}}-\Bigl\lfloor \frac{n'}{m_{u,v,\bsk}}\Bigr\rfloor\\
   &=\frac{n+r/b^*}{m_{u,v,\bszero}}-\Bigl\lfloor \frac{n+r/b^*}{m_{u,v,\bszero}}\Bigr\rfloor\\
   &=\frac{r}{b^*m_{u,v,\bszero}}+\frac{n}{m_{u,v,\bszero}}-\Bigl\lfloor \frac{n}{m_{u,v,\bszero}}\Bigr\rfloor\\
   &=\frac{r}{b^*m_{u,v,\bszero}}+\err_v
\end{align*}
  where  $\err_v=n/m_{u,v,\bszero}-\lfloor n/m_{u,v,\bszero}\rfloor$. Therefore, the normalized gain coefficients $G_{u,\bsk}(n')$ can be expressed as
  \begin{align*}
      G_{u,\bsk}(n')&=\frac1{n'}\prod_{j\in u}(\bj-1)^{-1} \sum_{v\subseteq u}H_{u,v}m_{u,v,\bsk}\err_v'(1-\err_v')\\
      &=\prod_{j\in u}(\bj-1)^{-1} \sum_{v\subseteq u}H_{u,v}\frac{b^*m_{u,v,\bszero}}{nb^*+r}\err_v'(1-\err_v')\\
      &=\prod_{j\in u}(\bj-1)^{-1} \sum_{v\subseteq u}H_{u,v}\frac{m_{u,v,\bszero}}{n+r/b^*}\Bigl(\err_v+\frac{r}{b^*m_{u,v,\bszero}}\Bigr)\Bigl(1-\err_v-\frac{r}{b^*m_{u,v,\bszero}}\Bigr)\\
      &=\prod_{j\in u}(\bj-1)^{-1} \sum_{v\subseteq u}H_{u,v}\frac{m_{u,v,\bszero}}{n+x}\biggl(\err_v(1-\err_v)+(1-2\err_v)\frac{x}{m_{u,v,\bszero}}-\frac{x^2}{m_{u,v,\bszero}^2}\biggr)
  \end{align*}
 where we have replaced $r/b^*$ with $x$. Let us extend the domain of $x$ to all real numbers in $[0,1]$. Our goal becomes to prove that $G_{u,\bsk}(n')$, as a function of $x$, is monotonic on $[0,1]$.

 First notice that because $H_{u,v}= \prod_{j\in v}b_j\prod_{j\in u-v}(-1)=(-1)^{|u-v|}m_{u,v,\bszero}$,
 \begin{align*}
      \sum_{v\subseteq u}H_{u,v}\frac{m_{u,v,\bszero}}{n+x}\frac{x^2}{m_{u,v,\bszero}^2}=\frac{x^2}{n+x}\sum_{v\subseteq u}(-1)^{|u-v|}=0.
 \end{align*}
 This allows us to rewrite $\prod_{j\in u}(\bj-1)G_{u,\bsk}(n')$ as
 \begin{align*}
 &\sum_{v\subseteq u}H_{u,v}\frac{m_{u,v,\bszero}}{n+x}\Bigl(\err_v(1-\err_v)+(1-2\err_v)\frac{x}{m_{u,v,\bszero}}\Bigr)\\
 =&    \sum_{v\subseteq u}H_{u,v}m_{u,v,\bszero}\err_v(1-\err_v)\frac{1}{n+x}+\sum_{v\subseteq u}H_{u,v}(1-2\err_v)\frac{x}{n+x}\\
 =&\frac{1}{n+x}\sum_{v\subseteq u}H_{u,v}\Bigl(m_{u,v,\bszero}\err_v(1-\err_v)-n(1-2\err_v)\Bigr)+\sum_{v\subseteq u}H_{u,v}(1-2\err_v).
 \end{align*}
 Monotonicity of $G_{u,\bsk}(n')$ follows from monotonicity of $1/(n+x)$ on $[0,1]$ and hence $G_{u,\bsk}(n')$ is maximized at either endpoint.
\end{proof}

\begin{theorem}\label{thm:supersetswin}
    For all  $1\le d<\infty$ and all nonempty $v\subseteq u\subseteq1{:}d$,
\begin{align}\label{eq:supersetswin}
\sup_{n\in\natu} G_{v,\bszero}(n) \le \sup_{n\in\natu}G_{u,\bszero}(n).
\end{align}
\end{theorem}
\begin{proof}
    It suffices to show the conclusion holds when $u-v$ is a single element $j^*$ and apply induction. Denote the maximizer of $G_{v,\bszero}(n)$ as $n^*$. Our goal is to show that
    \begin{equation}\label{eq:GvGu}
        \sup_{n\in\natu} G_{v,\bszero}(n)=G_{v,\bszero}(n^*)\leq G_{u,\bszero}(b_{j^*}n^*)\leq \sup_{n\in\natu}G_{u,\bszero}(n).
    \end{equation}
    For any subset $w\subseteq v$, we define $w_+=w\cup \{j^*\}$. Then 
    \begin{equation}\begin{aligned}\label{eq:identities}
    H_{u,w_+}&=b_{j^*}H_{u,w}, & H_{u,w}&=-H_{v,w},\\
    m_{u,w_+,\bszero}&=b_{j^*}m_{u,w,\bszero},\quad\text{and}\quad
    & m_{u,w,\bszero}&=m_{v,w,\bszero}.
\end{aligned}\end{equation}
       We also introduce $K(x)=x(1-x)$ 
    to simplify some expressions. Starting with  equation~\eqref{eq:Gexpr} and applying identities from~\eqref{eq:identities}, we get for any $n$ divisible by $b_{j^*}$ that
        \begin{align}\label{eq:GuGv}
        \widetilde G_{u,\bszero}(n)&=\sum_{w\subseteq u}H_{u,w}m_{u,w,\bszero}K(\err_w)\nonumber\\
        &=\sum_{w\subseteq v}H_{u,w_+}m_{u,w_+,\bszero}K\biggl(\frac{n}{m_{u,w_+,\bszero}}-\Bigl\lfloor\frac{n}{m_{u,w_+,\bszero}}\Bigr\rfloor\biggr)
        \nonumber\\                   &\quad
        +\sum_{w\subseteq v}H_{u,w}m_{u,w,\bszero}K\biggl(\frac{n}{m_{u,w,\bszero}}-\Bigl\lfloor\frac{n}{m_{u,w,\bszero}}\Bigr\rfloor\biggr)\nonumber\\
         &=\sum_{w\subseteq v}b^2_{j^*}H_{v,w}m_{v,w,\bszero}K\biggl(\frac{n}{b_{j^*}m_{v,w,\bszero}}-\Bigl\lfloor\frac{n}{b_{j^*}m_{v,w,\bszero}}\Bigr\rfloor\biggr)\nonumber\\ 
        &\quad-\sum_{w\subseteq v}H_{v,w}m_{v,w,\bszero}K\biggl(\frac{n}{m_{v,w,\bszero}}-\Bigl\lfloor\frac{n}{m_{v,w,\bszero}}\Bigr\rfloor\biggr) \nonumber\\
        &=b^2_{j^*}\widetilde G_{v,\bszero}(n/b_{j^*})-\widetilde G_{v,\bszero}(n).
    \end{align}
The corresponding normalized coefficient is
    \begin{align*}
        G_{u,\bszero}(n)&=\frac{1}{b_{j^*}n}\prod_{j\in u}(\bj-1)^{-1}\Bigl(b^2_{j^*}\widetilde G_{v,\bszero}(n/b_{j^*})-\widetilde G_{v,\bszero}(n)\Bigr)\nonumber\\
        &=\frac{b_{j^*}}{(b_{j^*}-1)n}\prod_{j\in v}(\bj-1)^{-1}\widetilde G_{v,\bszero}(n/b_{j^*})-\frac{1}{(b_{j^*}-1)b_{j^*}n}\prod_{j\in v}(\bj-1)^{-1}\widetilde G_{v,\bszero}(n)\nonumber\\
        &=\frac{b_{j^*}}{b_{j^*}-1}G_{v,\bszero}(n/b_{j^*})-\frac{1}{b_{j^*}-1}G_{v,\bszero}(n).
    \end{align*}
    Now, using the fact that $n^*$ is the maximizer of $G_{v,\bszero}(n)$
\begin{align*}G_{u,\bszero}(b_{j^*}n^*)
&=\frac{b_{j^*}}{b_{j^*}-1}G_{v,\bszero}(n^*)-\frac{1}{b_{j^*}-1}G_{v,\bszero}(b_{j^*}n^*)\\
&\ge G_{v,\bszero}(n^*).
\end{align*}
    The theorem immediately follows from equation~\eqref{eq:GvGu}.
\end{proof}

 \begin{theorem}\label{thm:betterupperbound}
     For scrambled Halton points the gains satisfy
\begin{align}\label{eq:gainbound2}
\sup_{n\in\natu} G_{u,\bsk}(n) \le \prod_{j\in u-\{j_m\}}\frac{b_j}{b_j-1}.
\end{align}
  for all $d\ge1$, all non-empty $u\subseteq1{:}d$ and all $\bsk\in\natu_0^u$, where $j_m=\arg\min_{j\in u} b_j$.
 \end{theorem}
 \begin{proof}
     According to Theorem~\ref{thm:onlykzero}, it suffices to prove the theorem for $\bsk=\bszero$. 
Therefore
 $$
\sup_{n\in\natu}G_{u,\bsk}(n)
=\sup_{n\in\natu}G_{u,\bszero}(n)
= \max_{1\le n\le \overline{m}}G_{u,\bszero}(n)
$$ where $\overline{m}=\overline{m}_{u,\bszero}=\overline{m}_{u,u,\bszero}
=\prod_{j\in u}b_j$ and the lower limit $1$ is $\underline{m}_{u,\bszero}=m_{u,\emptyset,\bszero}$.

    We proceed by induction on $|u|$.
     When $u$ only contains a single element $j$, a straightforward calculation shows for $1\leq n \leq b_j$ that
   $$G_{u,\bszero}(n)=\frac{b_j-n}{b_j-1}.$$
    So $\sup_{n\in\natu} G_{u,\bszero}(n)=G_{u,\bszero}(1)=1$ and the theorem is trivially true for $|u|=1$.
    
    Now for $|u|>1$, we assume that equation~\eqref{eq:gainbound2} holds for $v=u\setminus \{j^*\}$ with $j^*\neq j_m$ and then prove it holds for $u$. From equation~\eqref{eq:GuGv} and non-negativity of $\widetilde G_{v,\bszero}(n)$, 
    $$ \widetilde G_{u,\bszero}(n)\leq \sum_{w\subseteq v}b_{j^*}^2H_{v,w}m_{v,w,\bszero}K\biggl(\frac{n}{b_{j^*}m_{v,w,\bszero}}-\Bigl\lfloor\frac{n}{b_{j^*}m_{v,w,\bszero}}\Bigr\rfloor\biggr).$$ 
Let $m_{v,w,*}=b_{j^*}m_{v,w,\bszero}$ and 
$$G_{v,*}(n)=\frac1{n}\prod_{j\in v}(\bj-1)^{-1}\sum_{w\subseteq v}H_{v,w}m_{v,w,*}K\biggl(\frac{n}{m_{v,w,*}}-\Bigl\lfloor\frac{n}{m_{v,w,*}}\Bigr\rfloor\biggr).$$ 
We can proceed as in the proof of Theorem~\ref{thm:onlykzero} with $b^*$ replaced by $b_{j^*}$ and conclude
$$\sup_{n\in\natu} G_{v,*}(n)=\sup_{n\in\natu} G_{v,\bszero}(n)\leq \prod_{j\in v-\{j_m\}}\frac{b_j}{b_j-1}.$$ 
Hence
        \begin{align*}
        G_{u,\bszero}(n)&\leq \frac1{n}\prod_{j\in u}(\bj-1)^{-1} \sum_{w\subseteq v}b_{j^*}^2H_{v,w}m_{v,w,\bszero}K\biggl(\frac{n}{b_{j^*}m_{v,w,\bszero}}-\Bigl\lfloor\frac{n}{b_{j^*}m_{v,w,\bszero}}\Bigr\rfloor\biggr)\\
        &=\frac{b_{j^*}}{b_{j^*}-1}G_{v,*}(n)\\
        &\leq \prod_{j\in u-\{j_m\}}\frac{b_j}{b_j-1}
    \end{align*}
and the theorem follows from induction.
\end{proof}

\begin{corollary}\label{cor:halfprodratio}
For scrambled Halton points in dimension $d\ge1$
$$\sup_{n\ge1}\max_{u\subseteq1{:}d}\sup_{\bsk\in\natu_0^u}G_{u,\bsk}(n)\le \frac12\prod_{j=1}^d\frac{b_j}{b_j-1}.$$
\end{corollary}
\begin{proof}
The result holds for $d=1$.
For $d\ge2$,
\begin{align*}
\sup_{n\ge1}\max_{u\subseteq1{:}d}\sup_{\bsk\in\natu_0^u}G_{u,\bsk}(n)
=\sup_{n\in\natu}G_{1{:}d,\bszero}(n)
\le \prod_{j=2}^d\frac{b_j}{b_j-1}
=\frac12\prod_{j=1}^d\frac{b_j}{b_j-1}
\end{align*}
with the inequality coming from Theorem~\ref{thm:betterupperbound}.
\end{proof}

\begin{theorem}\label{thm:ratebound} 
  For the scrambled Halton points
\begin{align}\label{eq:ratebound}
\Gamma_d=\max_{\emptyset\ne u\subseteq1{:}d}
\sup_{\bsk\in \natu_0^u}\sup_{n\in\natu}G_{u,\bsk}(n) = O( \log(d))
\end{align}
as $d\to\infty$.
 \end{theorem}
\begin{proof}
First,
$\log(\Gamma_d) \le 
\sum_{j=1}^d\log(b_j/(b_j-1))$
where $b_j$ is the $j$'th prime number.
%Choose $\theta\in(0,1)$.
%By the prime number theorem, there exists $J_1(\theta)<\infty$
%such that $b_j > \underline{b}_j =j\log(j) + %\theta j\log(\log (j))$ 
%for all $j\ge J_1(\theta)$
For any $j\ge1$ we have $b_j>\underline{b}_j=j\log(j)$ by equation (3.12) of \cite{ross:scho:1962}.
For any $\epsilon>0$, a Taylor expansion gives
$$
\log\Bigl(\frac{b_j}{b_j-1}\Bigr)
 < \log\Bigl(\frac1{1-1/\underline{b}_j}\Bigr)
<\frac{1}{\underline{b}_j}
+\frac{1+\epsilon}{2\underline{b}_j^2}
$$
for all $j\ge J_1=J_1(\epsilon)$ for some $J_1<\infty$.
Then for all large enough $d$, some $J_2=J_2(\epsilon)\ge J_1(\epsilon)$ and some constants $c_{\epsilon}<c'_{\epsilon}<\infty$
\begin{align*}
\log( \Gamma_d)
&<c_{\epsilon}+\int_{J_2-1}^d\frac{1}{x\log(x)}\rd x
+\int_{J_2-1}^d\frac{1+\epsilon}{
2(x\log(x))^2}\rd x\\
&<c'_{\epsilon}+\int_{J_2-1}^d
\frac1{x\log(x)}\rd x
\\
&= \log(\log(d))+O(1)
\end{align*}
as $d\to\infty$.
Exponentiating this relationship establishes equation~\eqref{eq:ratebound}.
\end{proof}

\section{A  lower bound}\label{sec:lower}

Here we show that the gains cannot
be $O(\log(d)^{1-\epsilon})$ for any $\epsilon>0$. First we get a bound for the gain factor of any set $u$ that includes either $j=1$ or $j=2$.  This is equivalently about whether either 2 or 3 are among the primes $b_j$ for $j\in u$.

\begin{theorem}\label{thm:alowerbound}
For $1\le d<\infty$ and $u\subseteq 1{:}d$, if
$u\cap\{1,2\}\ne\emptyset$ then
%    Let $b^*=2$ or $3$.  For all  $1\le d<\infty$ and all nonempty $u\subseteq1{:}d$, if $\{b_j,j\in u\}$ contains $b^*$, then for all $\bsk\in\natu_0^u$
    $$\sup_{n\in\natu}G_{u,\bsk}(n)\geq \prod_{j\in u-\{j^*\}}\frac{b_j+1}{b_j}$$
    for any $\bsk\in\natu_0^u$
    where $j^*$ is any element of $u\cap\{1,2\}$.
\end{theorem}

\begin{proof}[Proof of Theorem~\ref{thm:alowerbound}]
According to Theorem~\ref{thm:onlykzero}, it suffices to prove the inequality for $\bsk=\bszero$. For $j^*\in u\cap\{1,2\}$, let $b_*=b_{j^*}$, $n^*=\prod_{j\in u,\bj\neq b_*} \bj$ and $V=\{v\subseteq u\mid j^*\in v\}$. Because $m_{u,v,\bszero}$ divides $n^*$ for any $v\notin V$, $\err_v=n^*/m_{u,v,\bszero}-\lfloor n^*/m_{u,v,\bszero}\rfloor=0$. Then equation~\eqref{eq:Gexpr} simplifies to
    \begin{align*}
        \widetilde G_{u,\bszero}(n^*)&= \sum_{v\in V}H_{u,v}m_{u,v,\bszero}K\biggl(\frac{n^*}{m_{u,v,\bszero}}-\Bigl\lfloor\frac{n^*}{m_{u,v,\bszero}}\Bigr\rfloor\biggr)\\
        &=\sum_{v\in V}(-1)^{|u-v|}\biggl(\prod_{j\in v}\bj^2 \biggr)K\biggl(\frac{1}{b_*}\prod_{j\in u-v}b_j-\Bigl\lfloor\frac{1}{b_*}\prod_{j\in u-v}b_j\Bigr\rfloor\biggr)
    \end{align*}
    where $K(x)=x(1-x)$.

    When $b_*=2$, because $\prod_{j\in u-v}b_j$ is odd, 
    $$K\biggl(\frac{1}{b_*}\prod_{j\in u-v}b_j-\Bigl\lfloor\frac{1}{b_*}\prod_{j\in u-v}b_j\Bigr\rfloor\biggr)=K\Bigl(\frac{1}{2}\Bigr)=\frac{1}{4}.$$
    When $b_*=3$, because $\prod_{j\in u-v}b_j$ an integer not divisible by 3, 
    $$K\biggl(\frac{1}{b_*}\prod_{j\in u-v}b_j-\Bigl\lfloor\frac{1}{b_*}\prod_{j\in u-v}b_j\Bigr\rfloor\biggr)=K\Bigl(\frac{1}{3}\Bigr) =K\Bigl(\frac{2}{3}\Bigr)=\frac{2}{9}.$$
    In either case,
    $$K\biggl(\frac{1}{b_*}\prod_{j\in u-v}b_j-\Bigl\lfloor\frac{1}{b_*}\prod_{j\in u-v}b_j\Bigr\rfloor\biggr)=\frac{b_*-1}{b_*^2}$$
    and the normalized coefficient equals
     \begin{align*}
        G_{u,\bszero}(n^*)
        &=\frac{1}{n^*}\prod_{j\in u}(\bj-1)^{-1}\sum_{v\in V}(-1)^{|u-v|}\biggl(\prod_{j\in v}\bj^2\biggr) \frac{b_*-1}{b_*^2}\\
        &=\prod_{j\in u,\bj\neq b_*}\frac{1}{\bj(\bj-1)}\sum_{v\in V}(-1)^{|u-v|}\prod_{j\in v,\bj\neq b_*}b_j^2 \\
        &=\prod_{j\in u,\bj\neq b_*}\frac{1}{\bj(\bj-1)}\prod_{j\in u,\bj\neq b_*}(b_j^2-1) \\
        &=\prod_{j\in u,\bj\neq b_*}\frac{\bj+1}{\bj}.
    \end{align*}   
    Hence 
        $$\sup_{n\in\natu}G_{u,\bszero}(n)\geq G_{u,\bszero}(n^*)=\prod_{j\in u,\bj\neq b^*}\frac{b_j+1}{b_j}.\qedhere$$
\end{proof}

For $d\ge2$ we divide $\prod_{j=1}^d(b_j+1)/b_j$ by either $3/2$ or $4/3$ and still get a lower bound. It follows that
$$
\Gamma_d \ge \frac34\prod_{j=1}^d\frac{b_j+1}{b_j}
$$
for $j\ge2$, while $\Gamma_1=1$.

\begin{corollary}\label{cor:betterlowerbound}
For any $\epsilon>0$ 
$$\Gamma_{1:d}=\sup_{n\ge1}\max_{u\subseteq1{:}d}\sup_{\bsk\in\natu_0^u}G_{u,\bsk}(n)$$
cannot be $O((\log d)^{1-\epsilon})$. 
\end{corollary}
\begin{proof}
First $1{:}d\cap\{1,2\}\ne\emptyset$, so Theorem~\ref{thm:alowerbound} gives
$\Gamma_{1:d}\ge \prod_{j=2}^d(b_j+1)/b_j$
(which is $1$ for $d=1$).
As in the proof of Theorem~\ref{thm:ratebound} we note that if $j\ge6$ then $b_j<j\log(j)+j\log(\log(j))$.  Then for $0<\epsilon'<\epsilon''<\epsilon$ and large enough $j$
$$
\log\Bigl(\frac{b_j+1}{b_j}\Bigr)\ge \frac{1-\epsilon'}{j\log(j)+j\log(\log(j))}
\ge \frac{1-\epsilon''}{j\log(j)}.
$$
 Using an integral lower bound like the one in the proof of Theorem~\ref{thm:alowerbound} we get
\begin{align*}
\log(\Gamma_{1:d})
\ge c + (1-\epsilon'')\log(\log(d))
\end{align*}
for some $c\in\real$.
After exponentiating, $\Gamma_{1:d}$ cannot be $O((\log d)^{1-\epsilon})$.
\end{proof}

\section{Conclusions}\label{sec:conclusions}

When we score RQMC methods by their worst case
variance relative to plain MC, then we find that
scrambled Halton points attain a much better
bound than scrambled Sobol' points do, while
retaining the $o(1/n)$ variance property.
This does not imply that scrambled Halton points
will be generally better than scrambled Sobol' points
in applications, because the integrands of interest
may not be ones where scrambled Sobol' points
perform poorly.   
It does make scrambled Halton points a useful approach
for settings where never performing much worse than
Monte Carlo is a priority.  We note that we could obtain a gain uniformly bounded in $d$ if we were to slightly increase the values $b_j$ in use. We do not recommend this as it would be detrimental to the equidistribution properties that QMC and RQMC are designed to produce.

\section*{Acknowledgments}
We thank Nabil Kahale who asked about methods with better
gain bounds than scrambled Sobol' points at MCM
2023, as well as C.\ D.\ Parada who raised the same
question in an email.
This work was supported by the National Science Foundation
under grant DMS-2152780.
\bibliographystyle{plain}
\bibliography{qmc}

\end{document}